\documentclass[12pt,oneside,english]{amsart}
\usepackage[T1]{fontenc}
\usepackage[latin9]{inputenc}
\usepackage{textcomp}
\usepackage{mathrsfs}
\usepackage{amsthm}
\usepackage{amstext}
\usepackage{amssymb}

\makeatletter

\newcommand{\lyxmathsym}[1]{\ifmmode\begingroup\def\b@ld{bold}
  \text{\ifx\math@version\b@ld\bfseries\fi#1}\endgroup\else#1\fi}

\numberwithin{equation}{section}
\numberwithin{figure}{section}
\theoremstyle{plain}
\newtheorem{thm}{\protect\theoremname}[section]
  \theoremstyle{remark}
  \newtheorem{rem}[thm]{\protect\remarkname}
  \theoremstyle{plain}
  \newtheorem{prop}[thm]{\protect\propositionname}

\makeatother

\usepackage{babel}
  \providecommand{\propositionname}{Proposition}
  \providecommand{\remarkname}{Remark}
\providecommand{\theoremname}{Theorem}

\begin{document}

\title{On operator-valued free convolution powers.}

\author{D. Shlyakhtenko}

\thanks{Research supported by NSF grant DMS-0900776.}

\address{Department of Mathematics, UCLA, Los Angeles, CA 90095, USA}

\email{shlyakht@math.ucla.edu}
\begin{abstract}
We give an explicit realization of the $\eta$-convolution power of
an $A$-valued distribution, as defined earlier by Anshelevich, Belinschi,
Fevrier and Nica. If $\eta:A\to A$ is completely positive and $\eta\geq\operatorname{id}$,
we give a short proof of positivity of the $\eta$-convolution power
of a positive distribution. Conversely, if $\eta\not\geq\operatorname{id}$,
and $s$ is large enough, we construct an $s$-tuple whose $A$-valued
distribution is positive, but has non-positive $\eta$-convolution
power. 
\end{abstract}
\maketitle

\section{Introduction.}

In this note, we investigate the question of positivity of $\eta$-free
convolution powers of an $A$-valued distribution. Such $\eta$-convolution
powers were introduced by Anshelevich, Belinschi, Fevrier and Nica
in \cite{anshelevich-etc-convolution-powers}, following a question
due to Bercovici. For $A=\mathbb{C}$ these correspond to the free
convolution powers considered by Nica and Speicher \cite{nica-speicher}.
The main theorem of \cite{anshelevich-etc-convolution-powers} is
a generalization (with a rather complicated proof) of a result from
\cite{nica-speicher}: if $\mu$ is a positive $A$-valued distribution
and $\eta:A\to A$ is a completely positive map so that $\eta-\operatorname{id}$
is completely positive, then the convolution power $\mu^{\boxplus\eta}$
is also positive. 

In the case that $A=\mathbb{C}$, a simple proof of this theorem exists:
for $t>1$, the convolution powers $\mu^{\boxplus t}$ are realized
(after some rescaling) in an explicit way by starting with some random
variable $X$ with distribution $\mu$ and compressing $X$ to a suitable
projection which is free from $X$ (see the appendix to \cite{nica-speicher}
by Voiculescu). 

We construct an explicit realization of the distribution of $\mu^{\boxplus\eta}$
as the distribution of $v^{*}Xv$, where $X$ has distribution $\mu$,
and $v$ is a certain specially constructed element free from $X$
with amalgamation over $A$ ($v$ is a multiple of isometry if $\eta(1)$
is a multiple of $1$). Positivity of the distribution of $v^{*}Xv$
is then immediate. The condition $\eta\geq\operatorname{id}$ appears
naturally in the construction of $v$. Our construction can be viewed
as a version of the proof of an explicit realization of $\mu^{\boxplus t}$
using free compression and the Fock space model given in \cite{shlyakht:compressions}.

In addition, we prove a converse to the theorem of Anshelevich et
al: if $\eta-\operatorname{id}$ is not completely positive, for $s$
large enough, there is an $s$-tuple which has a positive joint $A$-valued
distribution, but so that the $\eta$-convolution power of this distribution
is not positive.

\subsection*{Acknowledgements. }

The author is grateful to M. Anshelevich and S. Curran for several
discussions on this topic.

\subsection{$A$-valued distributions and realizability in a $C^{*}$-probability
space.}

We refer the reader to the book \cite{nica-speicher:amalg} for some
background on operator-valued free probability theory. Let $A$ be
a unital $C^{*}$-algebra. Recall that an $A$-probability space \cite{dvv:amalgfree,speicher:amalg}
is a unital $*$-algebra $B\supset A$ together with a conditional
expectation (i.e, an $A$-linear map) $E_{A}^{B}:B\to A$. For $X\in B$
and a non-commutative monomial $W=a_{0}Xa_{1}X\cdots Xa_{m}$, the
value $E_{A}^{B}(W)$ is called the non-commutative ($A$-valued)
moment of $X$; the map $\mu:W\mapsto E_{A}^{B}(W)$ is called the
($A$-valued) distribution of $X$. We say that an $A$-valued distribution
$\mu$ is \emph{positive} if it is possible to find some $C^{*}$-algbera
$B$, a positive $A$-linear map $E_{A}^{B}:B\to A$ and $X\in D$
so that the $A$-valued distributions of $X$ is exactly $\mu$. We
say that such an $X$ \emph{realizes} $\mu$. 

Positivity is an important property of an $A$-valued distribution;
for $A=\mathbb{C}$ positivity of a distribution corresponds to positivity
of a probability measure.

\subsection{Free cumulants $\omega_{k}^{X}$.\label{sec:freeCumulants}}

Associated to any $A$-valued distribution $\mu$ one has a sequence
of $\mathbb{C}$-multilinear maps $\omega_{k}^{\mu}:A^{k-1}\to A$
called the free cumulants of $X$ (here $\omega_{1}^{\mu}$ is simply
an element of $A$) \cite{dvv:amalgfree,speicher:amalg}. Let $Q$
be the universal algebra generated by elements $L^{\dagger}$, $L_{0},L_{1},\dots$
and $A$ subject to the relations:
\begin{eqnarray}
L_{0} & = & \omega_{1}^{\mu}\in A\subset Q\nonumber \\
L^{\dagger}a_{1}L^{\dagger}a_{2}L^{\dagger}a_{3}\cdots L^{\dagger}a_{k}L_{k} & = & \omega_{k+1}^{\mu}(a_{1},\dots,a_{k})\in A\subset Q.\label{eq:omegaRelations}
\end{eqnarray}
Finally, let $E_{A}^{Q}:Q\to A$ be determined by requiring that $E_{A}^{Q}\big|_{A}=\textrm{id}$
and that for any non-commutative monomial $W$ in elements of $A$,
$L^{\dagger}$, $L_{1}$, $L_{2}$, etc., $E_{A}^{Q}(W)=0$ unless
$W$ can be reduced to an element of $A$ using the relations (\ref{eq:omegaRelations}).
Then the sequence $\{\omega_{k}^{\mu}\}_{k\geq1}$ is uniquely determined
by the requirement that if we set $Y=L^{\dagger}+\sum_{k\geq0}L_{k}$,
the $A$-valued distribution of $Y$ is $\mu$.

\subsection{$\eta$-free convolution powers.}

Let $\mu$ be an $A$-valued distribution, and let $\eta:A\to A$
be a linear map. Define a new distribution $\mu^{\boxplus\eta}$ by
requiring that its free cumulants are given by $\omega_{k}^{\mu^{\boxplus\eta}}=\eta\circ\omega_{k}^{\mu}$.
This distribution is called, by definition, the $\eta$-convolution
power of $\mu$ (see equation (1.4) in \cite{anshelevich-etc-convolution-powers}).

\section{An explicit realization of the $\eta$-convolution powers.}

\subsection{Construction of the operator $v\in(C,E_{A}^{C})$\label{sub:Construction-of-V}.}

Let $A$ be a $C^{*}$-algebra, let $\psi:A\to A$ be a completely-positive
map, and let $\eta=\psi+\textrm{id}$. Let $\mathscr{H}$ be an $A,A$
Hilbert bimodule and $\xi\in\mathscr{H}$ be such that
\[
\langle\xi,a\xi\rangle_{\mathscr{H}}=\psi(a).
\]
Let $\mathscr{K}=\mathscr{H}\oplus A$ with the inner product $\langle h\oplus a,h'\oplus a'\rangle_{\mathscr{K}}=\langle h,h'\rangle_{\mathscr{H}}+a^{*}a'$.
Then $\mathscr{K}$ is an $A,A$ Hilbert bimodule with the diagonal
left and right actions of $A$. Finally let
\[
\mathscr{F}=A\oplus\mathscr{K}\oplus\mathscr{K}\otimes_{A}\mathscr{K}\oplus\cdots\oplus\mathscr{K}^{\otimes_{A}n}\oplus\cdots
\]
 be the full Fock space associated to $\mathscr{K}$ (see \cite{pimsner,speicher:amalg}).
We view $\mathscr{F}$ as an $A,A$-bimodule using the diagonal left
and right actions of $A$. We'll denote the left action of $A$ on
$\mathscr{F}$ by $\lambda$. 

Let us denote by $v$ the operator
\[
v:\mathscr{F}\to\mathscr{F},\quad\zeta_{1}\otimes\cdots\otimes\zeta_{n}\mapsto(\xi\oplus1)\otimes\zeta_{1}\otimes\cdots\otimes\zeta_{n}.
\]
Then an easy computation shows that
\[
v^{*}\lambda(a)v=\lambda(\eta(a)).
\]
Finally, for a bounded adjointable right $A$-linear operator $T:\mathscr{F}\to\mathscr{F}$,
set
\[
E_{A}^{C}(T)=\langle0\oplus1,T(0\oplus1)\rangle_{\mathscr{F}},
\]
where we regard $0\oplus1\in\mathscr{K}\subset\mathscr{F}$. Then
\begin{eqnarray*}
E_{A}^{C}(v\lambda(a)v^{*}) & = & \langle0\oplus1,v\lambda(a)v^{*}(0\oplus1)\rangle_{\mathscr{F}}\\
 & = & \langle0\oplus1,v\ a1\rangle_{\mathscr{F}}\\
 & = & \langle0\oplus1,(\xi a\oplus a)\rangle_{\mathscr{F}}=a.
\end{eqnarray*}
Letting $C=C^{*}(\lambda(A),v)$, we note that $(C,E_{A}^{C})$ is
an $A$-probability space. We'll also identify $A$ with $\lambda(A)$.
\begin{rem}
(i) Note that $v^{*}v=v^{*}\lambda(1)v=\eta(1)$. Thus if $\eta(1)=\alpha1$
with $\alpha\in\mathbb{R}$, then $\alpha^{-1/2}v$ is an isometry.
For general $\eta$, $v$ is not an isometry. (ii) In the case that
$A=\mathbb{C}$ and $\eta(a)=\lambda a$, $\lambda\in[1,+\infty)$,
the conditional expectation $E_{A}^{C}$ is non-tracial. Indeed, we
have that $E_{A}^{C}(vv^{*})=E_{A}^{C}(v\lambda(1)v^{*})=1$ but $E_{A}^{C}(v^{*}v)=E_{A}^{C}(\alpha)=\lambda$. 
\end{rem}

\subsection{The main result.}

Let $X\in(B,E_{A}^{B})$ and assume that $X$ has $A$-valued distribution
$\mu$. We will now compute the $A$-valued distribution of $\hat{X}=vXv^{*}$. 
\begin{prop}
\label{prop:computeCumulants}Assume that $\psi:A\to A$ is a completely-positive
map, and let $\eta=\psi+\operatorname{id}$ and let $v\in(C,E_{A}^{C})$
be as in \S\ref{sub:Construction-of-V}. Let $B$ be a $C^{*}$-algebra
and $E_{A}^{B}:B\to A$ be a positive $A$-linear map. Let $X=X^{*}\in B$
having distribution $\mu$. Consider $(M,E_{A}^{M})=(B,E_{A}^{B})*_{A}(C,E_{A}^{C})$,
let $\hat{X}=v^{*}Xv$, and let $\hat{\mu}$ be the distribution of
$\hat{X}$. Then the free cumulants $\omega_{k}^{\hat{\mu}}$ satisfy:
\begin{equation}
\omega_{k}^{\hat{\mu}}=\eta\circ\omega_{k}^{\mu}.\label{eq:cumulantsXhat}
\end{equation}
In particular, the $A$-valued distribution of $\hat{X}$ is positive. \end{prop}
\begin{proof}
Consider $(N,E_{A}^{N})=(B,E_{A}^{B})*_{A}(Q,E_{A}^{Q})$, and let
$Y=L^{\dagger}+\sum_{k\geq0}L_{k}\in Q$ be as in \S\ref{sec:freeCumulants}.
Let $\hat{Y}=v^{*}Y'v$. 

The $A$-valued distribution of $X'$ is the same as the $A$-valued
distribution of $Y'$.

Since $Q$ is free from $B$ with amalgamation over $A$, we may thus
assume \cite{nica-speicher-shlyakht:charactFreeness} that $L^{\dagger}$
and $L_{k}$ satisfy the relations
\[
L^{\dagger}b_{1}L^{\dagger}b_{2}L^{\dagger}b_{3}\cdots L^{\dagger}b_{k}L_{k}=\omega_{k+1}^{\mu}(E_{A}^{C}(b_{1}),\dots,E_{A}^{C}(b_{k})),\qquad b_{j}\in C
\]
and moreover for any monomial $W$ in elements of $C$ and $L^{\dagger},L_{1},L_{2},\dots$,
$E_{A}^{N}(W)=0$ 

Let $\hat{L}^{\dagger}=v^{*}L^{\dagger}v$ and $\hat{L}_{k}=v^{*}L_{k}v$.
Then we have:
\begin{eqnarray*}
\hat{L}^{\dagger}a_{1}\hat{L}^{\dagger}a_{2}\cdots\hat{L}^{\dagger}a_{k}\hat{L}_{k} & = & v^{*}L^{\dagger}va_{1}v^{*}L^{\dagger}va_{2}\cdots v^{*}L^{\dagger}va_{k}v^{*}L_{k}v\\
 & = & v^{*}\omega_{k+1}^{\mu}(E_{A}^{C}(va_{1}v^{*}),\dots,E_{A}^{C}(va_{k}v^{*}))v\\
 & = & v^{*}\omega_{k+1}^{\mu}(a_{1},\dots,a_{k})v\\
 & = & \eta(\omega_{k+1}^{\mu}(a_{1},\dots,a_{k})).
\end{eqnarray*}
Moreover, if $W$ is a non-commutative monomial in elements of $A$
and $\hat{L}^{\dagger},\hat{L}_{1},\hat{L}_{2},\dots$, then $E_{A}^{N}(W)=0$
unless $W$ can be reduced to an element of $A$ using this relation.
It then follows that if $\hat{\mu}$ is the distribution of $\hat{Y}$
(and is the same as the distribution of $\hat{X}$), then its free
cumulants are given by
\[
\omega_{k}^{\hat{\mu}}=\eta\circ\omega_{k}^{\mu}.
\]
This completes the proof.\end{proof}
\begin{thm}
Let $(B,E_{A}^{B})$ be an $A$-probability space and let $X\in B$
be a random variable whose $A$-valued distribution $\mu$ is positive.
Let $\eta:A\to A$ be completely-positive map so that $\eta-\operatorname{id}$
is completely positive. Let $\hat{X}=vXv^{*}$ be as in Proposition
\ref{prop:computeCumulants}. 

Then the distribution of $\hat{X}$ is the same as that of the $\eta$-convolution
power \cite{anshelevich-etc-convolution-powers} $X^{\boxplus\eta}$;
in other words, $vXv^{*}$ is an explicit realization of $\mu^{\boxplus\eta}$.

In particular, the $A$-valued distribution of $\mu^{\boxplus\eta}$
is also positive.\end{thm}
\begin{proof}
Let $\psi=\eta-\operatorname{id}$, so that $\eta=\psi+\operatorname{id}$.
Let $\hat{X}$ be as in Proposition \ref{prop:computeCumulants}.
By (\ref{eq:cumulantsXhat}) and \cite{anshelevich-etc-convolution-powers}
equation (1.4), the free cumulants of the distribution of $\hat{X}$
and of $\mu^{\boxplus\eta}$ are equal. Thus these $A$-valued distributions
are also equal. But $\hat{X}$ is explicitly realized in a $C^{*}$-probabilty
space and so its distribution is positive.
\end{proof}

\section{A converse.}

It is natural to ask whether the condition that $\eta-\operatorname{id}$
be completely-positive is necessary for $\eta$-convolution powers
to always remain positive (no matter what the initial distribution
is). We show that this is indeed the case if one considers joint distributions
of all $s$-tuples.
\begin{thm}
Assume that $\eta:A\to A$ is a completely positive map. Then $\eta-\operatorname{id}$
is completely-positive iff for every $s\geq1$ and every positive
$A$-valued distribution $\mu$ of an $s$-tuple, $\mu^{\boxplus\eta}$
is also positive.\end{thm}
\begin{proof}
There is a natural equivalence between $A$-valued distributions $\mu^{(X_{ij})}$
of $m^{2}$-tuples of variables $(X_{ij})_{i,j=1}^{m}$ and of the
$M_{m\times m}(A)$-valued distribution $\mu^{X}$ of the matrix $X=(X_{ij})$.
In fact, one easily obtains that the $\eta$-convolution power of
$\mu^{(X_{ij})}$ (defined by the requirement that the joint cumulants
are composed with $\eta$) correspond exactly to the $(\operatorname{id}_{m}\otimes\eta)$-convolution
powers of $\mu^{X}$. Thus positivity of $\mu^{\boxplus\eta}$ for
every $A$-valued distribution of an $s$-tuple is equivalent to positivity
of $\nu^{\boxplus(\operatorname{id}_{m}\otimes\eta)}$ for every $M_{m\times m}(A)$-valued
distribution of a single variable $\nu$. This completes the proof
of one direction of the theorem.

Assume now that there exists integer $m$ and $a\in M_{m\times m}(A)$,
$a>0$ so that $\eta_{m}(a)-a$ is not positive (here $\eta_{m}=\operatorname{id}\otimes\eta)$.
Let $\phi$ be a state on $M_{m\times m}(A)$ so that $\phi(\eta_{m}(a)-a)<0$.
Passing from $A$ to the enveloping von Neumann algebra $A^{**}$,
and from $a$ to a spectral projection of $a$, we may assume that
$a\in M_{m\times m}(A^{**})$ is projection and $\phi$ still satisfies
$\phi(\eta_{m}(a)-a)<-2\kappa<0$ for some fixed $\kappa>0$. By replacing
$\phi$ with a convex linear combination with a state that is strictly
positive on $a$ we may assume that $\phi(a)>0$ and sill $\phi(\eta_{m}(a))<\phi(a)-\kappa$. 

Let $\pi:M_{m\times m}(A^{**})\to B(H)$ be the GNS construction for
$\phi$ and denote by $\xi$ the associated cyclic vector in $H$.
Let $P\in B(H)$ be the rank one projection onto $\xi$. Denote by
$\hat{A}$ the $C^{*}$-algebra generated by $M_{m\times m}(A^{**})$
and $P$ inside $B(H)$. 

Choose $\delta>0$ so that $\delta<\kappa$. 

Note that $Tr(aPa)=\phi(a)$. Since $aPa$ is finite-rank, we can
find $N$ orthonormal vectors $\xi_{j}\in H$, $j=1,\dots,N$ so that
$\xi_{1}=\xi$ and $\left|Tr(aPa)-\sum\langle aPa\xi_{j},\xi_{j}\rangle\right|<\delta$
. Thus 
\[
\left|\sum\langle aPa\xi_{j},\xi_{j}\rangle-\phi(a)\right|<\delta.
\]
Let $\vartheta(x)=\frac{1}{N}\sum\langle x\xi_{j},\xi_{j}\rangle$
be a state on $\hat{A}$. Then $\vartheta(P)=\frac{1}{N}$ and so
\begin{equation}
\left|\frac{\vartheta(aPa)}{\vartheta(P)}-\phi(a)\right|<\delta.\label{eq:estimateOnVartheta}
\end{equation}

Let $X\in(B,\psi)$ be a self-adjoint random variable in a $\mathbb{C}$-valued
$C^{*}$-probability space $B$, and consider
\[
(C,\theta)=(\hat{A},\vartheta)*(B,\psi).
\]

Denote by $E=E_{\hat{A}}^{C}$ the conditional expectation from $C$
onto $\hat{A}$. If $\omega_{n}$ denotes the $n$-th scalar-valued
cumulant of $X$, then the $\hat{A}$-valued cumulants of $a^{1/2}Xa^{1/2}$
are given by
\[
\omega'_{n+1}(h_{1},\dots,h_{n})=\omega_{n+1}\ a\vartheta(ah_{1}a)\cdots\vartheta(ah_{n}a)a
\]
(see \cite{nica-speicher-shlyakht:charactFreeness}) and thus (recalling
that $a^{2}=a$) the $\hat{A}$-valued cumulants of the $\eta$-amplification
$Y$ of the distribution of $aXa$ are given by
\[
w''_{n+1}(h_{1},\dots,h_{n})=\eta_{m}(a)\ \omega_{n+1}\ \prod\vartheta(ah_{j}a).
\]
This means that the $\hat{A}$-valued cumulants of $PYP$ are given
by
\begin{eqnarray*}
\omega'''_{n+1}(h_{1},\dots,h_{n}) & = & P\eta_{m}(a)P\cdot\omega_{n+1}\cdot\prod\vartheta(aPh_{j}Pa)\\
 & = & \phi(\eta_{m}(a))\omega_{n+1}\cdot P\prod\vartheta(aPh_{j}Pa)
\end{eqnarray*}
(since $P\eta(a)P=\phi(\eta(a))P$). From this we see that the scalar-valued
cumulants of $PYP$ with respect to $\theta$ are given by
\[
\hat{\omega}_{n+1}=\vartheta(P)\phi(\eta_{m}(a))\vartheta(aPa)^{n}\cdot\omega_{n+1}.
\]
Let us finally set $Z=\vartheta(aPa)^{-1}PYP$. Then its scalar-valued
cumulants are given by
\[
\tilde{\omega}_{n+1}=\vartheta(P)\phi(\eta_{m}(a))\frac{\vartheta(aPa)^{n}}{\vartheta(aPa)^{n+1}}\omega_{n}=\frac{\vartheta(P)}{\vartheta(aPa)}\phi(\eta_{m}(a))\omega_{n+1}.
\]
Thus $\tilde{\omega}_{n}=\lambda\omega_{n}$ with
\[
\lambda=\frac{\phi(\eta_{m}(a))}{\vartheta(aPa)/\vartheta(P)}<\frac{\phi(a)-\kappa}{\vartheta(aPa)/\vartheta(P)}.
\]
By (\ref{eq:estimateOnVartheta}) and our choice of $\delta<\kappa$,
we conclude that
\[
\lambda<\frac{\phi(a)-\kappa}{\phi(a)-\delta}<1.
\]
In other words, the $\mathbb{C}$-valued distribution of $Z$ is the
same as that of the $\lambda$-convolution power of the $\mathbb{C}$-valued
distribution of $X$ for some $\lambda<1$.

Assume now for contradiction that the laws of all $\operatorname{id}\otimes\eta$-convolution
powers of $M_{m\times m}(A)$ are positive. Choose $a_{k}\in M_{m\times m}(A)$
so that $a_{k}\to a$ weakly and $\sup\Vert a_{k}\Vert<\infty$. Then
if we set $Y_{k}$ to be the $\eta$-convlution power of the distribution
of $a_{k}Xa_{k}$ (which is positive by our assumption), and $Z_{k}=\vartheta(aPa)^{-1}PY_{k}P$,
then we see that $Z_{k}\to Z$ in moments. But then positivity of
distributions of $Z_{k}$ implies positivity of the distribution of
$Z$.

To summarize, assuming that the $\eta$-convolution power of distribution
of every $M_{m\times m}(A)$-valued distribution is positive, we concluded
that the law of the scalar-valued distribution of $X$ admits a positive
$\lambda$-convolution power for some $\lambda<1$. But this is not
always possible: for example, we could start with $X$ having as distribution
the sum of two equal point masses; it is known that this distribution
admits no $\lambda$-convolution power if $\lambda<1$. \end{proof}

\end{document}